\documentclass[12pt]{amsart}
\usepackage{graphicx}
\usepackage{amsmath, amscd, amssymb, amsthm}
\usepackage[subnum]{cases}
\usepackage{changepage}
\usepackage{xcolor}
\usepackage{marginnote}
\usepackage{hyperref}
\usepackage{cleveref}
\usepackage[all]{xy}
\usepackage{tabularx}
\usepackage{ltablex}
\usepackage{longtable}
\usepackage[T1]{fontenc}

\newtheorem{theorem}{Theorem}[section]
\newtheorem{lemma}[theorem]{Lemma}

\newtheorem{proposition}[theorem]{Proposition}
\newtheorem{corollary}[theorem]{Corollary}
\newtheorem{conjecture}[theorem]{Conjecture}

\theoremstyle{definition}
\newtheorem{definition}[theorem]{Definition}
\newtheorem{remark}[theorem]{Remark}
\numberwithin{equation}{section}
\newtheorem{example}[theorem]{Example}

\newtheorem{setting}[theorem]{Setting}

\usepackage{fancyhdr}
\begin{document}

\normalfont

\title{Topologization and Functional Analytification I: Intrinsic Morphisms of Commutative Algebras}
\author{Xin Tong}

\maketitle

\begin{abstract}
\rm Eventually after Dieudonn\'e-Grothendieck, we give intrinsic definitions of \'etale, lisse and non-ramifi\'e morphisms for general adic rings and general locally convex rings. And we investigate the corresponding \'etale-like, lisse-like and non-ramifi\'e-like morphisms for general $\infty$-Banach, $\infty$-Born\'e and $\infty$-ind-Fr\'echet $\infty$-rings  and $\infty$-functors into $\infty$-groupoid (as in the work of Bambozzi-Ben-Bassat-Kremnizer) in some intrinsic way by using the corresponding infinitesimal stacks and crystalline stacks. The two directions of generalization will intersect at Huber's book in the strongly noetherian situation. 
\end{abstract}

\footnotetext[1]{Version: Feb 21 2021.}


\tableofcontents


\section{Introduction}

\subsection{Main Consideration}

\noindent Scholze's diamond is actually very general notion beyond the corresponding perfectoid spaces, partially because it contains the corresponding diamantine spaces after Hansen-Kedlaya. This point of view certainly gives the motivation for this notion from Hansen-Kedlaya. Therefore one could regard to some extent diamantine spaces as giving some (maybe better to say more ring theoretic) analogs of the corresponding Scholze's diamonds, moreover they behave as if they are perfectoid spaces. Similar discussion could be made to suosperfectoid space, which should be more 'perfectoid' generalization.\\

\indent Hansen-Kedlaya \cite{HK} have given the definition of naive \'etale morphisms among any Tate Huber pairs namely these are locally composites of the rational localizations and finite \'etale morphisms (very importantly with strongly sheafy domains and targets). This is because one definitely believes that correct notion of \'etale morphism should admit such admissible decomposition and factorization. However what should be the correct intrinsic one has not been given in full detail yet. In the significant strongly sheafy situation, we are going to try to answer this question as proposed in \cite[Appendix 5]{Ked1}. In this paper, we try to study the corresponding properties of the corresponding naive \'etale morphisms along the ideas of \cite[Appendix 5]{Ked1} and \cite{HK}. The goal is to accurately characterize the corresponding naive \'etale morphisms in some intrinsic way.\\

\indent Sheafiness plays a very crucial role in the discussion above. However suppose we do not have to worry about the sheafiness at all (in fact in some sense we really do not have to worry about this at all by the work of Clausen-Scholze \cite{CS} and Bambozzi-Kremnizer \cite{BK}), then one might want to believe that the robust definitions could be made even more robust. Therefore we investigate the corresponding morphisms of the corresponding ring objects where sheafiness could be replaced by $\infty$-sheafiness (namely sheafiness up to higher homotopy) after \cite{BBBK} and \cite{BK}. One should be able to consider Clausen-Scholze's foundation \cite{CS} as well, however we will mainly focus on the $\infty$-locally convex objects in \cite{BBBK} and \cite{BK}, as in the corresponding schematic situation in \cite{Lu1}, \cite{Lu2}, \cite{TV1} and \cite{TV2}. We consider the corresponding interesting approaches through the corresponding formal and PD completions just as in the $\infty$-schematic situation in \cite{R} which is very related to the corresponding Drinfeld's stacky construction \cite{Dr1} and \cite{Dr2} by using the \v{C}ech-Alexander complex on the corresponding crystalline cohomology and prismatic cohomology.\\

\indent The current list of definitions will be established for discrete $\mathbb{E}_\infty$-ring objects and $\mathbb{E}_\infty$-ring objects in suitable locally convex $\infty$-categories after after Bambozzi-Ben-Bassat-Kremnizer \cite{BBBK}:\\

\indent D1. Localized intrinsic \'etale morphisms of open mapping Huber rings;\\
\indent D2. Localized intrinsic \'etale morphisms of open mapping adic Banach rings;\\ 
\indent D3. Localized intrinsic lisse morphisms of open mapping Huber rings;\\
\indent D4. Localized intrinsic lisse morphisms of open mapping adic Banach rings;\\
\indent D5. Localized intrinsic non-ramifi\'e morphisms of open mapping Huber rings;\\
\indent D6. Localized intrinsic non-ramifi\'e morphisms of open mapping adic Banach rings;\\
\indent $\infty$1. De Rham intrinsic \'etale-like morphisms of $\infty$-analytic functors;\\
\indent $\infty$2. De Rham intrinsic lisse-like morphisms of $\infty$-analytic functors;\\ 
\indent $\infty$3. De Rham intrinsic non-ramifi\'e-like morphisms of $\infty$-analytic functors;\\
\indent $\infty$4. PD (crystalline) intrinsic \'etale-like morphisms of $\infty$-analytic functors;\\
\indent $\infty$5. PD (crystalline) intrinsic lisse-like morphisms of $\infty$-analytic functors;\\ 
\indent $\infty$6. PD (crystalline) intrinsic non-ramifi\'e-like morphisms of $\infty$-analytic functors.\\

\indent Certainly for general locally convex spaces producing nice ring structures we really have to be very precise and accurate in any sorts of characterization. However we have not unfortunately achieve this due to some very subtle issues, mainly coming from the corresponding issues in very general functional analytification. That being otherwise all said, we still actually could literally talk about the desired definitions for simplicial noetherian Banach rings in certain situations.

\subsection{Further Consideration}

Our ultimate goal is certainly to study the corresponding geometric sites (\'etale, pro-\'etale, crystalline and prismatic \cite{SGAIV},\cite{Gro1},\cite{Sch1},\cite{KL1},\cite{KL2},\cite{BS}, \cite{Dr1},\cite{Dr2}) and the corresponding cohomologies (\'etale, pro-\'etale, crystalline and prismatic) for really general $\infty$-analytic spaces (possibly also noncommutative analogs of those in \cite{KR1}) over $\mathbb{F}_1$ and try to apply to the locally noetherian situations, the strongly noetherian situations (such as in \cite{G1}, \cite{GL}), the strongly sheafy situations under the foundation of $\infty$-locally convex spaces (as in \cite{HK}, \cite{KL1} and more general situations), although our very beginning corresponding motivation for this article is an attempt to answer some questions in \cite[Appendix A5]{Ked1}.

\

\section{Affinoid Morphisms of Huber Rings}

\noindent We start with the discussion on the corresponding intrinsic definition of \'etale morphisms.

\begin{setting}
We start with an analytic uniform Huber pair $(A,A^+)$. And we will consider the category of all such rings. We assume the corresponding completeness for the Huber pairs.
\end{setting}

\begin{definition}\mbox{\bf{(Hansen-Kedlaya \cite[Definition 5.1]{HK})}}
We call a map of Huber rings $(A,A^+)\rightarrow (B,B^+)$ naive \'etale after \cite[Definition 5.1]{HK} if it admit a factorization into rational localizations and finite \'etale morphisms. Here we assume $(A,A^+)$ is strong sheafy and we assume that $(B,B^+)$ is strongly sheafy. \footnote{Certainly one needs to be careful since we are now considering more general context than \cite{HK} without assuming the corresponding Tateness.}	
\end{definition}

\begin{definition}\mbox{\bf{(Kedlaya \cite[Definition A5.2]{Ked1})}}\label{definition2.3}
Recall from \cite[Definition A5.2]{Ked1}, we have the corresponding affinoid morphism from any strongly sheafy Huber ring $A$, namely a morphism $A\rightarrow B$, such that $B$ admits some surjective covering from $A\left<T_1,...,T_d\right>$ and through this map we have that $B$ is a stably-pseudocoherent sheaf over $A\left<T_1,...,T_d\right>$ and the corresponding ring $B$ is assumed to be sheafy \footnote{Cetainly one needs to be more careful since this is also slightly different from the original definition \cite[Definition A5.2]{Ked1}, thanks Professor Kedlaya for telling me this should be better. We want to mention that this is a quite subtle point around the sheafiness (see \cite[Theorem 1.4.20]{Ked1}), the point here is that we do not know the kernel of an affinoid morphism is closed or not, if it is closed then we could keep the knowledge that $B$ being stably-pseudocoherent is equivalent to $B$ being sheafy. However if this is not closed, then we do not have this sort of equivalence to our knowledge.}.

\end{definition}

\indent The belief (as proposed in \cite[Problem A5.3, Problem A5.4]{Ked1}) is that somehow the corresponding affinoid morphisms in the definition should be directly used in the corresponding definitions of lisse morphisms and unramified morphisms, as well as certainly the \'etale morphisms. To investigate this kind of idea, we are going to first investigate the corresponding naive \'etale morphisms along this idea.\\

\begin{lemma}
Let $f_1:\Gamma_1\rightarrow \Gamma_2$ and $f_2:\Gamma_2\rightarrow \Gamma_3$ be two affinoid morphisms, then the composition $f_2\circ f_1$ is also affinoid.	
\end{lemma}

\begin{proof}
Straightforward.	
\end{proof}


\begin{lemma}\mbox{\bf{(Kedlaya)}} \label{lemma2.5}
For any standard binary rational localization of $A$ with respect to $f,g\in A$, suppose we know that there are two surjective morphisms:
\begin{align}
s_1:A\left<\frac{f}{g}\right>\left<T_1,...,T_n\right>\rightarrow B\left<\frac{f}{g}\right>,\\
s_2:A\left<\frac{g}{f}\right>\left<T_1,...,T_{n}\right>\rightarrow B\left<\frac{g}{f}\right>.	
\end{align}
Then we have that there is a surjective morphism:
\begin{align}
s:A\left<T_1,...,T_{n'}\right>\rightarrow B.
\end{align}	
\end{lemma}

\begin{proof}
The following argument is due to Kedlaya \footnote[1]{Thanks Professor Kedlaya for mentioning the similarity of this to the corresponding locality of morphisms of finite type as in Grothendieck's EGA I and II.}, we work out it for the convenience of the readers. First, we have the following short exact sequence:
\[
\xymatrix@C+0pc@R+0pc{
&0 \ar[r] \ar[r] \ar[r] &B \ar[r] \ar[r] \ar[r] &B\left<\frac{f}{g}\right>\bigoplus B\left<\frac{g}{f}\right> \ar[r] \ar[r] \ar[r] &B\left<\frac{f}{g},\frac{g}{f}\right> \ar[r] \ar[r] \ar[r] &0.
}
\]
Take any $b\in B$, and use the notation $(b_1,b_2)$ for the image in the middle. By the surjectivity of the maps $s_1,s_2$ we have that there exist some element $a_1\in A\left<\frac{f}{g}\right>\left<T_1,...,T_n\right>$ and some element $a_2\in A\left<\frac{f}{g}\right>\left<T_1,...,T_n\right>$ such that we have:
\begin{align}
s_1(a_1)=b_1,\\
s_2(a_2)=b_2.	
\end{align}
With more explicit expression we have the following:
\begin{align}
s_1(\sum_{i_1,...,i_n}\sum_{i} a_1^{i,i_1,...,i_n}u^iT_1^{i_1}...T_n^{i_n})=\sum_i b_1^iu^i,\\
s_2(\sum_{i_1,...,i_n}\sum_{i} a_2^{i,i_1,...,i_n}v^iT_1^{i_1}...T_n^{i_n})=\sum_i b_2^iv^i,	
\end{align}	
under the corresponding presentations up to liftings:
\begin{align}
B\left<\frac{f}{g}\right>=B\left<u\right>/(gu-f),\\
B\left<\frac{g}{f}\right>=B\left<v\right>/(fv-g).\\	
\end{align}
Then to finish we only have to take some finite sum in the summation to make approximation. We first claim that such finite sum approximation and modification will not change the corresponding surjectivity of the map $s_1$ and $s_2$. Namely for each $k=1,2$ the map $s_k$ will maintain surjective once we modify the image of $T_1,...,T_n$ infinitesimally around some neighbourhood $U$ of $0$, in other words it will maintain to be surjective even if we set $s_k(T_1),...,s_k(T_n)$ to be $x_1,...,x_n$ whenever $x_1-s_k(T_1),...,x_n-s_k(T_n)$ lives in the neighbourhood $U$, and moreover we have that the corresponding modification could be assumed to take $T_i$ to $x_i$ with $i=1,...,n$. By open mapping, we have that the corresponding lifts of the corresponding differences $x_1-s_k(T_1),...,x_n-s_k(T_n)$ could be made to be living in some arbitrarily chosen neighbourhood $V$ of $0$. Then we only have to consider the following map factoring through the corresponding map $s_k$:
\begin{align}
h: A_k\left<T_1,...,T_n\right>&\rightarrow A_k\left<T_1,...,T_n\right>\\
	T_i&\mapsto T_i+\mathrm{lifts~of}~x_i-s_k(T_k)
\end{align}
where $A_1$ is the ring $A\left<\frac{f}{g}\right>$ while we have $A_2$ is the ring $A\left<\frac{g}{f}\right>$, which basically proves the claim. Then this will indicate that one can find some joint finite subset $T:=\{T_1,...,T_{n'}\}$ for $B\left<\frac{f}{g}\right>$ and $B\left<\frac{f}{g}\right>$ such that the modified  
\begin{align}
s_1:A\left<\frac{f}{g}\right>\left<T_1,...,T_{n'}\right>\rightarrow B\left<\frac{f}{g}\right>,\\
s_2:A\left<\frac{g}{f}\right>\left<T_1,...,T_{n'}\right>\rightarrow B\left<\frac{g}{f}\right>,	
\end{align}
are basically surjective and they fit into the following commutative diagram:
\[\tiny
\xymatrix@C+0pc@R+3pc{
0 \ar[r] \ar[r] \ar[r] &A\left<T_1,...,T_{n'}\right> \ar[d] \ar[d] \ar[d] \ar[r] \ar[r] \ar[r] &A\left<\frac{f}{g}\right>\left<T_1,...,T_{n'}\right>\bigoplus A\left<\frac{g}{f}\right>\left<T_1,...,T_{n'}\right> \ar[d] \ar[d] \ar[d]\ar[r] \ar[r] \ar[r] &A\left<\frac{f}{g},\frac{g}{f}\right>\left<T_1,...,T_{n'}\right> \ar[d] \ar[d] \ar[d]\ar[r] \ar[r] \ar[r] &0\\
0 \ar[r] \ar[r] \ar[r] &B  \ar[r] \ar[r] \ar[r] &B\left<\frac{f}{g}\right>\bigoplus B\left<\frac{g}{f}\right> \ar[r] \ar[r] \ar[r] &B\left<\frac{f}{g},\frac{g}{f}\right> \ar[r] \ar[r] \ar[r] &0,
}
\]
where the middle and the rightmost vertical arrows are surjective. Then claim is then that the left vertical one is also surjective. The kernels $K_1\oplus K_2$ in the middle is mapped surjectively to the kernel $K_{12}$ of the rightmost vertical map. So the snake lemma will force the cokernel of the left vertical arrow to  be zero which shows the corresponding exactness at the corresponding location $?$ in the following commutative diagram:
\[\tiny
\xymatrix@C+0pc@R+3pc{
 &0 \ar[d] \ar[d] \ar[d]   &0 \ar[d] \ar[d] \ar[d]  &0 \ar[d] \ar[d] \ar[d] & \\
0 \ar[r] \ar[r] \ar[r]& K\ar[d] \ar[d] \ar[d]\ar[r] \ar[r] \ar[r]&K_1\bigoplus K_2 \ar[d] \ar[d] \ar[d]\ar[r] \ar[r] \ar[r]&K_{12} \ar[d] \ar[d] \ar[d]\ar[r] \ar[r] \ar[r]&\\
0 \ar[r] \ar[r] \ar[r] &A\left<T_1,...,T_{n'}\right> \ar[d] \ar[d] \ar[d] \ar[r] \ar[r] \ar[r] &A\left<\frac{f}{g}\right>\left<T_1,...,T_{n'}\right>\bigoplus A\left<\frac{g}{f}\right>\left<T_1,...,T_{n'}\right> \ar[d] \ar[d] \ar[d]\ar[r] \ar[r] \ar[r] &A\left<\frac{f}{g},\frac{g}{f}\right>\left<T_1,...,T_{n'}\right> \ar[d] \ar[d] \ar[d]\ar[r] \ar[r] \ar[r] &0\\
0 \ar[r] \ar[r] \ar[r] &B  \ar[d]^? \ar[d] \ar[d]\ar[r] \ar[r] \ar[r] &B\left<\frac{f}{g}\right>\bigoplus B\left<\frac{g}{f}\right> \ar[d] \ar[d] \ar[d]\ar[r] \ar[r] \ar[r] &B\left<\frac{f}{g},\frac{g}{f}\right> \ar[d] \ar[d] \ar[d]\ar[r] \ar[r] \ar[r] &0\\
 &0   &0  &0 &
}
\]
where $K_1,K_2,K_{12}$ are pseudocoherent, which implies that the corresponding module $K$ is also pseudocoherent.

\end{proof}

\begin{proposition} \label{proposition2.5}
Any naive \'etale morphism is affinoid.
\end{proposition}

\begin{proof}
First let $f:(A,A^+)\rightarrow (B,B^+)$ be any naive \'etale morphism. Then locally this is basically composition of the corresponding rational localizations and finite \'etale maps. Locally rational localizations involved are actually affinoid, and locally the corresponding finite \'etale maps from strongly sheafy rings will have strongly sheafy target, which will imply that locally finite \'etale maps are affinoid. Then this could be globalized to force the global map $f$ to be affinoid. The properties of factoring through a surjection globally could be proved by glueing local ones through \cref{lemma2.5}, by considering \cite[Proposition 2.4.20]{KL1}. And globally the corresponding ring $B$ is stably-pseudocoherent over $A\left<T_1,...,T_n\right>$ for some $n$ since this is a local property. 
\end{proof}

\indent Therefore we have proved that the corresponding \'etale maps in the corresponding naive sense is actually affinoid in the above sense. Therefore it is now natural to try to find the corresponding properties which may completely characterize the corresponding naive \'etale morphisms which are affinoid.

\indent Certainly we may have the corresponding conjectures that all the naive \'etale morphisms will satisfy the corresponding properties of algebraically \'etale ones (such as in \cite[Chapitre 17]{EGAIV4}, \cite[Tag 00U1]{SP}). We now discuss the corresponding completed cotangent complex after Huber \cite[1.6.2]{Hu1}. Recall for our current $B$ the corresponding completed differential $\Omega^1_{B/A,\mathrm{topo}}$ (see \cite[1.6.2]{Hu1} for the construction for any $f$-adic rings in the noetherian setting). Therefore we consider the corresponding topological naive cotangent complex:
\begin{align}
\tau_{\leq 1}\mathbb{L}_{B/A,\mathrm{topo}}	
\end{align}
for any naive \'etale map $f:(A,A^+)\rightarrow (B,B^+)$. We now discuss the construction without the corresponding strongly noetherian requirement in our current situation. First we know that $B$ is of topologically finite type over $A$:
\begin{align}
B=A\left<X_1,...,X_n\right>_{T_1,...,T_n}/I.	
\end{align}
Then we could first define the topological free differentials:
\begin{align}
\Omega^1:=A\left<X_1,...,X_n\right>_{T_1,...,T_n}dX_1+...+A\left<X_1,...,X_n\right>_{T_1,...,T_n}dX_n.	
\end{align}
Then we have:
\begin{align}
\Omega^1_{B/A,\mathrm{topo}}:=	\Omega^1\slash (I\bigcup d(I))\Omega^1.
\end{align}

Here everything is assumed to be basically complete with respect to the corresponding natural topology. Namely we need to take the corresponding completion always with respect to the corresponding induced topology. Certainly here $\Omega^1$ is already complete due to the fact that it is finitely projective. Recall that a map $f:\Gamma_1\rightarrow \Gamma_2$ is called \'etale in the scheme theory if the naive cotangent complex (truncated and could be regarded as an $\infty$-module spectrum) is quasi-isomorphic to zero. The corresponding underlying complex reads:
\[
\xymatrix@C+0pc@R+0pc{
[I/I^2\ar[r] \ar[r] \ar[r] &\Omega^1_{\Gamma_2/\Gamma_1,\mathrm{topo}}].
}
\]

\indent In the situation where we consider $A\rightarrow B$ is affinoid, the corresponding ideal $I$ is actually stably-pseudocoherent over $A\left<X_1,...,X_n\right>$. It is peudocoherent by the corresponding two out of three property. The stability holds locally, so we have the case. And if the morphism if furthermore naive \'etale then we have $I/I^2$ is also stably-pseudocoherent, see \cref{lemma4.8}. 

\begin{remark}
Note that we are considering the very general and complicated non-noetherian situation, modules will need to be endowed with the natural topology and complete, although finite projective modules are complete automatically. This will have nontrivial things to do with the corresponding definition of $\Omega^1_{B/A,\mathrm{topo}}$.	
\end{remark}

\indent One can actually generalize the corresponding full cotangent complexes and derived de Rham complexes to this topological context following \cite{III1}, \cite{III2} and \cite{B1}. First for the corresponding topological cotangent complex we consider the following definition (note that we have to assume the corresponding topologically finite type condition). We start with the corresponding algebraic ones for $B^h=A[X_1,...,X_n]_{T_1,...,T_n}/I$, under the topologization we have the corresponding derived cotangent complex:
\begin{align}
\mathbb{L}_{B^h/A,\mathrm{alg}},	
\end{align}
by taking the usual algebraic one. Then we take the corresponding completion with respect to the corresponding topologization which gives rise to the following topological one:
\begin{align}
\mathbb{L}_{B/A,\mathrm{topo}}.	
\end{align}

We define the corresponding de Rham complex in the following parallel way. What is happen is that consider the presentation $B^h=A[X_1,...,X_n]_{T_1,...,T_n}/I$ which gives rise to the corresponding algebraic de Rham complex:
\[
\xymatrix@C+0pc@R+0pc{
0\ar[r]\ar[r]\ar[r] &B^h \ar[r]\ar[r]\ar[r] & \Omega^1_{B^h/A,\mathrm{alg}}\ar[r]\ar[r]\ar[r] & \Omega^2_{B^h/A,\mathrm{alg}}  \ar[r]\ar[r]\ar[r] &...\ar[r]\ar[r]\ar[r] & \Omega^\bullet_{B^h/A,\mathrm{alg}} \ar[r]\ar[r]\ar[r] &..., 
}
\]
which will give rise to the corresponding topological one if we take the corresponding completion induced from the subset $T_1,...,T_n$:
\[
\xymatrix@C+0pc@R+0pc{
0\ar[r]\ar[r]\ar[r] &B \ar[r]\ar[r]\ar[r] & {\Omega}^1_{B^h/A,\mathrm{topo}}\ar[r]\ar[r]\ar[r] & {\Omega}^2_{B^h/A,\mathrm{topo}}  \ar[r]\ar[r]\ar[r] &...\ar[r]\ar[r]\ar[r] & {\Omega}^\bullet_{B^h/A,\mathrm{topo}} \ar[r]\ar[r]\ar[r] &.... 
}
\]	

From our construction for ${\Omega}^1_{B^h/A,\mathrm{topo}}$, one can actually define:
\begin{align}
{\Omega}^{\bullet,\mathrm{f}}_{B^h/A,\mathrm{topo}}:=\bigoplus_{i_1,...,i_\bullet\in \{1,...,n\}}A\left<X_1,...,X_n\right>_{T_1,...,T_n} dX_{i_1}\wedge dX_{i_2}\wedge...\wedge dX_{i_\bullet}
\end{align}
and then define:
\begin{align}
{\Omega}^{\bullet}_{B/A,\mathrm{topo}}:=\left(\bigoplus_{i_1,...,i_\bullet\in \{1,...,n\}}A\left<X_1,...,X_n\right>_{T_1,...,T_n} dX_{i_1}\wedge dX_{i_2}\wedge...\wedge dX_{i_\bullet}\right)\slash\\
\left((I\bigcup dI \bigcup d^\bullet I)\bigoplus_{i_1,...,i_\bullet\in \{1,...,n\}}A\left<X_1,...,X_n\right>_{T_1,...,T_n} dX_{i_1}\wedge dX_{i_2}\wedge...\wedge dX_{i_\bullet}\right),
\end{align}
after taking suitable completion when needed \footnote[1]{As in \cite{B1} and \cite{GL} where one takes the corresponding derived $p$-completion out from the algebraic cotangent complex and the corresponding derived algebraic de Rham complex.}.

\

\section{Affinoid Morphisms of Banach Rings}

\noindent We now consider the parallel situation of Banach rings.

\begin{setting}
We start with a uniform adic Banach ring $(A,A^+)$ in the general sense of \cite{KL1} and \cite{KL2} (without assumption on the topologically nilpotent units being existing, but we assume this is open mapping). And we will consider the category of all such rings. We assume the corresponding completeness as well.
\end{setting}

\begin{definition}\mbox{\bf{(Hansen-Kedlaya \cite[Definition 5.1]{HK})}}
We call a map of adic Banach rings $(A,A^+)\rightarrow (B,B^+)$ naive \'etale after \cite[Definition 5.1]{HK} if it admit a factorization into rational localizations and finite \'etale morphisms. Here we assume $(A,A^+)$ is strong sheafy and we assume that $(B,B^+)$ is sheafy.	
\end{definition}

\begin{remark}
Certainly this is in more general situation than the corresponding context of \cite{HK}.	
\end{remark}

\begin{definition}\mbox{\bf{(Kedlaya \cite[Definition A5.2]{Ked1})}}\label{definition2.3}
Recall from \cite[Definition A5.2]{Ked1}, we have the corresponding affinoid morphism from any strongly sheafy adic Banach ring $A$, namely a morphism $A\rightarrow B$, such that $B$ admits some surjective covering from $A\left<T_1,...,T_d\right>$ and through this map we have that $B$ is a stably-pseudocoherent sheaf over $A\left<T_1,...,T_d\right>$ and we assume that $(B,B^+)$ is strongly sheafy.

\end{definition}

\begin{remark}
Of course the corresponding foundation is not the same but parallel to such situation we are considering now, however it is definitely reasonable and parallel to follow \cite[Appendix A5]{Ked1} to give the definition here.	
\end{remark}

\indent The belief (as proposed in \cite[Problem A5.3, Problem A5.4]{Ked1}) is that somehow the corresponding affinoid morphisms in the definition should be directly used in the corresponding definitions of lisse morphisms and unramified morphisms, as well as certainly the \'etale morphisms. To investigate this kind of idea, we are going to first investigate the corresponding naive \'etale morphisms along this idea.\\

\begin{lemma}\mbox{\bf{(Kedlaya)}} \label{lemma2.5}
For any standard binary rational localization of $A$ with respect to $f,g\in A$, suppose we know that there are two surjective morphisms:
\begin{align}
s_1:A\left<\frac{f}{g}\right>\left<T_1,...,T_n\right>\rightarrow B\left<\frac{f}{g}\right>,\\
s_2:A\left<\frac{g}{f}\right>\left<T_1,...,T_{n}\right>\rightarrow B\left<\frac{g}{f}\right>.	
\end{align}
Then we have that there is a surjective morphism:
\begin{align}
s:A\left<T_1,...,T_{n'}\right>\rightarrow B.
\end{align}	
\end{lemma}

\begin{proof}
The following argument is due to Kedlaya, we work out it for the convenience of the readers. And this is the corresponding Banach analog of corresponding parallel result in the Huber ring situation. First, we have the following short exact sequence:
\[
\xymatrix@C+0pc@R+0pc{
&0 \ar[r] \ar[r] \ar[r] &B \ar[r] \ar[r] \ar[r] &B\left<\frac{f}{g}\right>\bigoplus B\left<\frac{g}{f}\right> \ar[r] \ar[r] \ar[r] &B\left<\frac{f}{g},\frac{g}{f}\right> \ar[r] \ar[r] \ar[r] &0.
}
\]
Take any $b\in B$, and use the notation $(b_1,b_2)$ for the image in the middle. By the surjectivity of the maps $s_1,s_2$ we have that there exist some element $a_1\in A\left<\frac{f}{g}\right>\left<T_1,...,T_n\right>$ and some element $a_2\in A\left<\frac{f}{g}\right>\left<T_1,...,T_n\right>$ such that we have:
\begin{align}
s_1(a_1)=b_1,\\
s_2(a_2)=b_2.	
\end{align}
With more explicit expression we have the following:
\begin{align}
s_1(\sum_{i_1,...,i_n}\sum_{i} a_1^{i,i_1,...,i_n}u^iT_1^{i_1}...T_n^{i_n})=\sum_i b_1^iu^i,\\
s_2(\sum_{i_1,...,i_n}\sum_{i} a_2^{i,i_1,...,i_n}v^iT_1^{i_1}...T_n^{i_n})=\sum_i b_2^iv^i,	
\end{align}	
under the corresponding presentations up to liftings:
\begin{align}
B\left<\frac{f}{g}\right>=B\left<u\right>/(gu-f),\\
B\left<\frac{g}{f}\right>=B\left<v\right>/(fv-g).\\	
\end{align}
Then to finish we only have to take some finite sum in the summation to make approximation. We first claim that such finite sum approximation and modification will not change the corresponding surjectivity of the map $s_1$ and $s_2$. Namely for each $k=1,2$ the map $s_k$ will maintain surjective once we modify the image of $T_1,...,T_n$ infinitesimally around some neighbourhood $U$ of $0$, in other words it will maintain to be surjective even if we set $s_k(T_1),...,s_k(T_n)$ to be $x_1,...,x_n$ whenever $\|x_1-s_k(T_1)\|\leq \delta,...,\|x_n-s_k(T_n)\|\leq \delta$ for some prescribed constant $\delta<1$ and moreover we have that the corresponding modification could be assumed to take $T_i$ to $x_i$ with $i=1,...,n$. By open mapping in this current context, we have that the corresponding lifts of the corresponding differences $x_1-s_k(T_1),...,x_n-s_k(T_n)$ could be made to be living in some arbitrarily chosen neighbourhood $V$ of $0$ namely, we can find lifts $y_1,...,y_n$ of these differences such that
\begin{align}
\|y_1\|<1,...,\|y_n\|<1.	
\end{align}
Then we only have to consider the following map factors through the corresponding map $s_k$:
\begin{align}
h: A_k\left<T_1,...,T_n\right>&\rightarrow A_k\left<T_1,...,T_n\right>\\
	T_i&\mapsto T_i+\mathrm{lifts~of}~x_i-s_k(T_k)
\end{align}
where $A_1$ is the ring $A\left<\frac{f}{g}\right>$ while we have $A_2$ is the ring $A\left<\frac{g}{f}\right>$, which basically proves the claim. Then this will indicate that one can find some joint finite subset $T:=\{T_1,...,T_{n'}\}$ for $B\left<\frac{f}{g}\right>$ and $B\left<\frac{f}{g}\right>$ such that the modified  
\begin{align}
s_1:A\left<\frac{f}{g}\right>\left<T_1,...,T_{n'}\right>\rightarrow B\left<\frac{f}{g}\right>,\\
s_2:A\left<\frac{g}{f}\right>\left<T_1,...,T_{n'}\right>\rightarrow B\left<\frac{g}{f}\right>,	
\end{align}
are basically surjective and they fit into the following commutative diagram:
\[\tiny
\xymatrix@C+0pc@R+3pc{
0 \ar[r] \ar[r] \ar[r] &A\left<T_1,...,T_{n'}\right> \ar[d] \ar[d] \ar[d] \ar[r] \ar[r] \ar[r] &A\left<\frac{f}{g}\right>\left<T_1,...,T_{n'}\right>\bigoplus A\left<\frac{g}{f}\right>\left<T_1,...,T_{n'}\right> \ar[d] \ar[d] \ar[d]\ar[r] \ar[r] \ar[r] &A\left<\frac{f}{g},\frac{g}{f}\right>\left<T_1,...,T_{n'}\right> \ar[d] \ar[d] \ar[d]\ar[r] \ar[r] \ar[r] &0\\
0 \ar[r] \ar[r] \ar[r] &B  \ar[r] \ar[r] \ar[r] &B\left<\frac{f}{g}\right>\bigoplus B\left<\frac{g}{f}\right> \ar[r] \ar[r] \ar[r] &B\left<\frac{f}{g},\frac{g}{f}\right> \ar[r] \ar[r] \ar[r] &0,
}
\]
where the middle and the rightmost vertical arrows are surjective. Then claim is then that the left vertical one is also surjective. The kernels $K_1\oplus K_2$ in the middle is mapped surjectively to the kernel $K_{12}$ of the rightmost vertical map. So the snake lemma will force the cokernel of the left vertical arrow to  be zero which shows the corresponding exactness at the corresponding location $?$ in the following commutative diagram:
\[\tiny
\xymatrix@C+0pc@R+3pc{
 &0 \ar[d] \ar[d] \ar[d]   &0 \ar[d] \ar[d] \ar[d]  &0 \ar[d] \ar[d] \ar[d] & \\
0 \ar[r] \ar[r] \ar[r]& K\ar[d] \ar[d] \ar[d]\ar[r] \ar[r] \ar[r]&K_1\bigoplus K_2 \ar[d] \ar[d] \ar[d]\ar[r] \ar[r] \ar[r]&K_{12} \ar[d] \ar[d] \ar[d]\ar[r] \ar[r] \ar[r]&\\
0 \ar[r] \ar[r] \ar[r] &A\left<T_1,...,T_{n'}\right> \ar[d] \ar[d] \ar[d] \ar[r] \ar[r] \ar[r] &A\left<\frac{f}{g}\right>\left<T_1,...,T_{n'}\right>\bigoplus A\left<\frac{g}{f}\right>\left<T_1,...,T_{n'}\right> \ar[d] \ar[d] \ar[d]\ar[r] \ar[r] \ar[r] &A\left<\frac{f}{g},\frac{g}{f}\right>\left<T_1,...,T_{n'}\right> \ar[d] \ar[d] \ar[d]\ar[r] \ar[r] \ar[r] &0\\
0 \ar[r] \ar[r] \ar[r] &B  \ar[d]^? \ar[d] \ar[d]\ar[r] \ar[r] \ar[r] &B\left<\frac{f}{g}\right>\bigoplus B\left<\frac{g}{f}\right> \ar[d] \ar[d] \ar[d]\ar[r] \ar[r] \ar[r] &B\left<\frac{f}{g},\frac{g}{f}\right> \ar[d] \ar[d] \ar[d]\ar[r] \ar[r] \ar[r] &0\\
 &0   &0  &0 &
}
\]
where $K_1,K_2,K_{12}$ are pseudocoherent, which implies that the corresponding module $K$ is also pseudocoherent.

\end{proof}

\begin{lemma}
Let $f_1:\Gamma_1\rightarrow \Gamma_2$ and $f_2:\Gamma_2\rightarrow \Gamma_3$ be two affinoid morphisms, then the composition $f_2\circ f_1$ is also affinoid.	
\end{lemma}

\begin{proof}
Straightforward.	
\end{proof}

\begin{proposition} \label{proposition2.5}
Any naive \'etale morphism is affinoid.
\end{proposition}

\begin{proof}
See \cref{proposition2.5}.
\end{proof}

\indent Therefore as in the corresponding Huber pair situation, we have proved that the corresponding \'etale maps in the corresponding naive sense is actually affinoid in the above sense. Therefore it is now natural to try to find the corresponding properties which may completely characterize the corresponding naive \'etale morphisms which are affinoid.

\indent Certainly we may have the corresponding conjectures that all the naive \'etale morphisms will satisfy the corresponding properties of algebraically \'etale ones (such as in \cite[Chapitre 17]{EGAIV4}, \cite[Tag 00U1]{SP}). We now discuss the corresponding Banach completed cotangent complex after Huber \cite[1.6.2]{Hu1}. Recall for our current $B$ the corresponding completed differential $\Omega^1_{B/A,\mathrm{topo}}$ (see \cite[1.6.2]{Hu1} for the construction for any $f$-adic rings). Certainly in the context of adic Banach rings we have the parallel completed version of the differentials by taking Banach completion, which is in some trival way the current situation. Therefore we consider the corresponding topological naive cotangent complex:
\begin{align}
\tau_{\leq 1}\mathbb{L}_{B/A,\mathrm{topo}}	
\end{align}
for any naive \'etale map $f:(A,A^+)\rightarrow (B,B^+)$. We consider the construction without the corresponding strongly noetherian requirement in our current situation. First we know that $B$ is of topologically finite type over $A$:
\begin{align}
B=A\left<X_1,...,X_n\right>_{T_1,...,T_n}/I.	
\end{align}
Then we could first define the Banach free differentials:
\begin{align}
\Omega^1:=A\left<X_1,...,X_n\right>_{T_1,...,T_n}dX_1+...+A\left<X_1,...,X_n\right>_{T_1,...,T_n}dX_n.	
\end{align}
Then we have:
\begin{align}
\Omega^1_{B/A,\mathrm{topo}}:=	\Omega^1\slash (I\bigcup d(I))\Omega^1.
\end{align}

Here everything is assumed to be basically complete with respect to the corresponding natural topology. Namely we need to take the corresponding completion always with respect to the corresponding induced norms. Certainly here $\Omega^1$ is already complete due to the fact that it is finitely projective. Recall that a map $f:\Gamma_1\rightarrow \Gamma_2$ is called \'etale in the scheme theory if the naive cotangent complex (truncated and could be regarded as an $\infty$-module spectrum) is quasi-isomorphic to zero. The corresponding underlying complex reads:
\[
\xymatrix@C+0pc@R+0pc{
[I/I^2\ar[r] \ar[r] \ar[r] &\Omega^1_{\Gamma_2/\Gamma_1,\mathrm{topo}}].
}
\]

\indent In the situation where we consider $A\rightarrow B$ is affinoid, the corresponding ideal $I$ is actually stably-pseudocoherent over $A\left<T_1,...,T_n\right>$. It is peudocoherent by the corresponding two out of three property. The stability holds locally, so we have the case. And if the morphism if furthermore naive \'etale then we have $I/I^2$ is also stably-pseudocoherent, see \cref{lemma4.8}. 

\begin{remark}
Note that we are considering the very general and complicated non-noetherian situation, modules will need to be endowed with the natural topology coming from the Banach structures on the Banach rings and complete, although in this situation as well finite projective modules are complete automatically. This will have nontrivial things to do with the corresponding definition of $\Omega^1_{B/A,\mathrm{topo}}$.	
\end{remark}

\indent In the Banach world, one can actually generalize the corresponding full cotangent complexes and de Rham complex to this context. First for the corresponding topological cotangent complex we consider the following definition (note that we have to assume the corresponding topologically finite type condition). We start with the corresponding algebraic ones for $B^h=A[X_1,...,X_n]_{T_1,...,T_n}/I$, under the topologization we have the corresponding derived cotangent complex:
\begin{align}
\mathbb{L}_{B^h/A,\mathrm{alg}},	
\end{align}
by taking the usual algebraic one. Then we take the corresponding completion with respect to the corresponding topologization which gives rise to the following topological one:
\begin{align}
\mathbb{L}_{B/A,\mathrm{topo}}.	
\end{align}

We define the corresponding de Rham complex in the following way parallely. What is happen is that consider the presentation $B^h=A[X_1,...,X_n]_{T_1,...,T_n}/I$ which gives rise to the corresponding algebraic de Rham complex:
\[
\xymatrix@C+0pc@R+0pc{
0\ar[r]\ar[r]\ar[r] &B^h \ar[r]\ar[r]\ar[r] & \Omega^1_{B^h/A,\mathrm{alg}}\ar[r]\ar[r]\ar[r] & \Omega^2_{B^h/A,\mathrm{alg}}  \ar[r]\ar[r]\ar[r] &...\ar[r]\ar[r]\ar[r] & \Omega^\bullet_{B^h/A,\mathrm{alg}} \ar[r]\ar[r]\ar[r] &..., 
}
\]
which will give rise to the corresponding topological one if we take the corresponding completion under $(\|.\|,\mathrm{Ban})$ induced from the subset $T_1,...,T_n$:
\[
\xymatrix@C+0pc@R+0pc{
0\ar[r]\ar[r]\ar[r] &B^h_{\|.\|,\mathrm{Ban}} \ar[r]\ar[r]\ar[r] & \Omega^1_{B^h/A,\mathrm{alg},\|.\|,\mathrm{Ban}}\ar[r]\ar[r]\ar[r] & \Omega^2_{B^h/A,\mathrm{alg},\|.\|,\mathrm{Ban}}  \ar[r]\ar[r]\ar[r] &...\ar[r]\ar[r]\ar[r] & \Omega^\bullet_{B^h/A,\mathrm{alg},\|.\|,\mathrm{Ban}} \ar[r]\ar[r]\ar[r] &..., \\
0\ar[r]\ar[r]\ar[r] &B \ar[r]\ar[r]\ar[r] & {\Omega}^1_{B^h/A,\mathrm{topo}}\ar[r]\ar[r]\ar[r] & {\Omega}^2_{B^h/A,\mathrm{topo}}  \ar[r]\ar[r]\ar[r] &...\ar[r]\ar[r]\ar[r] & {\Omega}^\bullet_{B^h/A,\mathrm{topo}} \ar[r]\ar[r]\ar[r] &.... 
}
\]	
From our construction for ${\Omega}^1_{B^h/A,\mathrm{topo}}$, one can actually define:
\begin{align}
{\Omega}^{\bullet,\mathrm{f}}_{B^h/A,\mathrm{topo}}:=\bigoplus_{i_1,...,i_\bullet\in \{1,...,n\}}A\left<X_1,...,X_n\right>_{T_1,...,T_n} dX_{i_1}\wedge dX_{i_2}\wedge...\wedge dX_{i_\bullet}
\end{align}
and then define:
\begin{align}
{\Omega}^{\bullet}_{B/A,\mathrm{topo}}:=\left(\bigoplus_{i_1,...,i_\bullet\in \{1,...,n\}}A\left<X_1,...,X_n\right>_{T_1,...,T_n} dX_{i_1}\wedge dX_{i_2}\wedge...\wedge dX_{i_\bullet}\right)\slash\\
\left((I\bigcup dI \bigcup d^\bullet I)\bigoplus_{i_1,...,i_\bullet\in \{1,...,n\}}A\left<X_1,...,X_n\right>_{T_1,...,T_n} dX_{i_1}\wedge dX_{i_2}\wedge...\wedge dX_{i_\bullet}\right),
\end{align}
after taking suitable completion under Banach norms when needed. One can also follow the construction in \cite{III1}, \cite{III2}, \cite{B1} and \cite{GL} to first consider the corresponding polynomial resolution $P_\bullet$ for $B$, then consider the corresponding algebraic derived de Rham complex $\Omega^*_{P_\bullet/A,\mathrm{alg}}$, then take the corresponding Banach completion to produce the corresponding topological one $\Omega^*_{P_\bullet/A,\mathrm{topo}}$. Then as in \cite{III1}, \cite{III2}, \cite{GL} and \cite{B1} around analytic derived $p$-adic de Rham complex we can take the corresponding suitable derived filtered completion to get the complex $\widehat{\Omega}^{\bullet}_{B/A,\mathrm{topo}}$ (certainly we need to consider Banach version of some filtered derived category of simplicial Banach rings). In the pro-\'etale site theoretic setting for rigid spaces (namely the corresponding p-complete context) this recovers the corresponding construction in \cite{GL}. Recall in \cite{GL} the corresponding analytic derived $p$-adic de Rham complex is constructed by first define the integral version ${\Omega}^{\bullet}_{B_0/A_0,\mathrm{topo}}$, and then take the colimit throughout all such rings of definition, and invert $p$, and then take the filtered completion. 

\begin{remark}
Certainly after this previous discussion we can construct the corresponding Banach derived de Rham complex, Banach cotangent complex and Banach Andr\'e-Quillen homology for any morphism $A\rightarrow B$ of Banach rings admissible in our situation. Recall from \cite{III1}, \cite{III2}, \cite{B1} we have the corresponding algebraic $p$-adic derived cotangent complex:
\begin{align}
\Omega^1_{A[B]^\bullet/A}\otimes_{A[B]^\bullet}B
\end{align}
where $A[B]^\bullet$ is just the corresponding standard cofibrant replacement (in the topology theoretic language) of $B/A$. Then we take the corresponding derived completion \footnote{The derived completion in our Banach situation is the derived Banach completion which for instance could happen by using the 'completion' functor from $\mathrm{Simp}(\mathrm{Ind}(\mathrm{NormSets}))$ to $\mathrm{Simp}(\mathrm{Ind}(\mathrm{BanachSets}))$ literally in \cite{BBBK}.} under the induced norm to achieve the corresponding topological one:
\begin{align}
\mathbb{L}_{B/A,\mathrm{topo}}:=(\Omega^1_{A[B]^\bullet/A}\otimes_{A[B]^\bullet}B)^\wedge_{\|.\|}.
\end{align}
We have the corresponding algebraic $p$-adic derived de Rham complex:
\begin{align}
\Omega^\bullet_{A[B]^\bullet/A}
\end{align}
where $A[B]^\bullet$ is just the corresponding standard cofibrant replacement (in the topology theoretic language) of $B/A$. Then we take the corresponding completion under the induced norm to achieve the corresponding topological one:
\begin{align}
\mathbb{\mathrm{dR}}_{B/A,\mathrm{topo}}:=(\Omega^\bullet_{A[B]^\bullet/A})^\wedge_{\|.\|},
\end{align}
carrying certain filtration $\mathrm{Fil}^*_{\mathbb{\mathrm{dR}}_{B/A,\mathrm{topo}}}$, which allows one take the corresponding filtered completion to achieve the corresponding final object:
\begin{align}
\mathbb{\mathrm{dR}}^\wedge_{B/A,\mathrm{topo}}:=(\Omega^\bullet_{A[B]^\bullet/A})^\wedge_{\|.\|,\mathrm{Fil}^*_{\mathbb{\mathrm{dR}}_{B/A,\mathrm{topo}}}}.
\end{align}
\end{remark}

\

\section{Naive \'Etale Morphisms and Intrinsic \'Etale Morphisms}

\indent As discussed above we now study the corresponding properties of naive \'etale morphisms aiming at the corresponding characterization of the corresponding correct definitions of intrinsic \'etale morphisms. We now assume the corresponding analyticity of the adic rings. 

\begin{lemma}
Let $f:(A,A^+)\rightarrow (B,B^+)$ be any rational localization map. Then we have that $B$ is of topologically finite type.
\end{lemma}

\begin{proof}
Straightforward.	
\end{proof}

\begin{lemma}
Let $f:(A,A^+)\rightarrow (B,B^+)$ be any rational localization map. Then we have that $B$ is affinoid over $A$ in the sense \cref{definition2.3}.
\end{lemma}

\begin{proof}
See the proof of \cref{proposition2.5}.	
\end{proof}

\begin{lemma}
Let $f:(A,A^+)\rightarrow (B,B^+)$ be any rational localization map. Then we have
\begin{align}
\tau_{\leq 1}\mathbb{L}_{B/A,\mathrm{topo}}	
\end{align}
is quasi-isomorphic to zero.
\end{lemma}

\begin{proof}
It suffices to reduce to standard binary localization such as simple Laurent or balanced localization, where we give a proof in the case of simple Laurent one:
\begin{align}
A\rightarrow A\{T\}/(T-f)	
\end{align}
for some $f\in A$. Then we have that actually the corresponding topological cotangent complex will be the corresponding completion of the corresponding algebraic ones (see \cite[Proposition 1.6.3]{Hu1}).	The corresponding quasi-isomorphism could be defined directly. One just considers the following algebraic differential map:
\begin{align}
I=(T-f) \rightarrow A[T]/(T-f) dT\\	
\end{align}
which is actually surjective since for any:
\begin{align}
\sum_{i\geq 0}a_iT^i+g(T)(T-f) dT \in A[T]/(T-f) dT	
\end{align}
one takes the corresponding integration of:
\begin{align}
\int_{f}^T	\sum_{i\geq 0}a_iT^i+g(T)(T-f) dT &= \int_{f}^T	\sum_{i\geq 0}a_iT^i+(\sum_{i\geq 0} g_iT^i) (T-f) dT\\
&= \int_{f}^T	\sum_{i\geq 0}a_iT^i+\int_{f}^T (\sum_{i\geq 0} g_iT^i)(T-f) dT \\
&= \int_{f}^T	\sum_{i\geq 0}a_iT^i+ \int_f^T \sum_{i\geq 0} g_iT^{i+1}- f\int_f^T \sum_{i\geq 0} g_iT^{i}\\
&=\sum_{i\geq 0}a_i\frac{1}{i+1}T^{i+1}|_{f}^T+\sum_{i\geq 0}g_i\frac{1}{i+2}T^{i+2}|_{f}^T-f \sum_{i\geq 0}g_i\frac{1}{i+1}T^{i+1}|_{f}^T  \\
&= (*)(T-f),
\end{align}
where we only have finite sums here since we are considering the corresponding topologized polynomial. For this algebraic map, we have that kernel is $(T-f)^2$, for instance consider:
\begin{align}
d((T-f)h(T))=0,	
\end{align}
we will have:
\begin{align}
h(T)+(T-f)h'(T)dT=0,	
\end{align}
which implies that the image of $(T-f)h(T)$ lives in the corresponding quotient $A[T]/(T-f)dT$. Therefore we have that the topologized (not complete yet) cotangent complex:
\[
\xymatrix@C+0pc@R+0pc{
[(T-f)/(T-f)^2\ar[r] \ar[r] \ar[r] &A[T]/(T-f)dT],
}
\]
which is quasi-isomorphic to zero. Then we take the corresponding completion with respect to the corresponding topology induced from $A$ we have the desired result.

\end{proof}

\begin{lemma}
Let $f:(A,A^+)\rightarrow (B,B^+)$ be any finite \'etale map. Then we have that $B$ is of topologically finite type.
\end{lemma}

\begin{proof}
Since we have that $B$ is affinoid over $A$ by \cref{proposition2.5}.	
\end{proof}

\begin{lemma}
Let $f:(A,A^+)\rightarrow (B,B^+)$ be any finite \'etale map. Then we have that $B$ is affinoid over $A$ in the sense \cref{definition2.3}.
\end{lemma}

\begin{proof}
See the proof of \cref{proposition2.5}.	
\end{proof}

\begin{lemma}
Let $f:(A,A^+)\rightarrow (B,B^+)$ be any finite \'etale map. Then we have 
\begin{align}
\tau_{\leq 1}\mathbb{L}_{B/A,\mathrm{topo}}	
\end{align}
is quasi-isomorphic to zero.	
\end{lemma}

\begin{proof}
This is basically nontrivial due to the fact that we are discussing the corresponding topological cotangent complex. However, one takes the corresponding finite presentation $B=A[T_1,...,T_n]/(f_1,...,f_p)$ (note that in fact that we have that $B$ is finite over $A$) since we are considering a finite \'etale map, which will realize the corresponding desired algebraic cotangent complex:
\[
\xymatrix@C+0pc@R+0pc{
[I/I^2\ar[r] \ar[r] \ar[r] &\Omega^1_{B/A}].
}
\]
For the topological situation we have that $B=A\{T_1,...,T_n\}/(f_1,...,f_p)$ by taking the corresponding completion. Again note that this means that actually the corresponding $A$-algebra $B$ is still finite over $A$ (since that is the very assumption). Therefore we could have the chance to right $B$ as just $\oplus_{i} Ae_i$ this is basically inducing the same differential module $\bigoplus_j(\bigoplus_{i} Ae_i) dT_j$ in both the corresponding topological setting and the algebraic setting. Then one could get the corresponding desired topological cotangent complex which is quasi-isomorphic to zero.
\end{proof}

\indent Then we consider the corresponding local composition:

\begin{lemma}
Let $f:(A,A^+)\rightarrow (B,B^+)$ be any naive \'etale map. Then we have that $B$ is of topologically finite type.
\end{lemma}

\begin{proof}
Since we have that $B$ is affinoid over $A$ by \cref{proposition2.5}.	
\end{proof}

\begin{lemma} \label{lemma4.8}
Let $f:(A,A^+)\rightarrow (B,B^+)$ be any naive \'etale morphism. Then we have 
\begin{align}
\tau_{\leq 1}\mathbb{L}_{B/A,\mathrm{topo}}	
\end{align}
is quasi-isomorphic to zero locally with respect to the corresponding rational localization.	
\end{lemma}

\begin{proof}
Locally we have that any naive \'etale morphism takes the corresponding truncated cotangent complex to be trivialized, by the corresponding composition properties of the cotangent complex \cite[Tag 08PN]{SP}.   
\end{proof}

\indent In order to globalize the picture one has to work harder. First we have the following:

\begin{proposition}
Let $f:(A,A^+)\rightarrow (B,B^+)$ be any naive \'etale morphism. Then $f$ is affinoid, of topologically finite type.	
\end{proposition}

\begin{proof}
This is by \cref{proposition2.5} for the affinoidness, which implies the corresponding second property. 
\end{proof}

\begin{definition}
We now define localized \footnote{However this could actually be globalized easily.} intrinsic \'etale morphism to be a morphism $f:(A,A^+)\rightarrow (B,B^+)$ which is affinoid with strongly sheafy target, and locally (with respect to the rational localization) the corresponding truncated topological cotangent complex is quasi-isomorphic to zero.	
\end{definition}


\indent We now consider the corresponding intrinsic \'etale morphisms of the corresponding special adic spaces after \cite{HK}:

\begin{setting}
We now consider the three special adic spaces after \cite{HK}, they are the corresponding strongly sheafy adic spaces, the corresponding sousperfectoid adic spaces and the corresponding diamantine adic spaces. We will use the notations $T,S,D$ to denote them in general respectively.	
\end{setting}

\indent The corresponding categories of strongly sheafy adic spaces, sousperfectoid spaces and diamantine adic spaces are nice enough since at least we have well-defined notion of naive \'etale morphisms (which is certainly the correct one) and furthermore well-defined \'etale sites.

\begin{definition}
For strongly sheafy adic spaces, a morphism $T_1\rightarrow T_2$ is called localized intrinsic \'etale if locally on $T_1$ this is localized intrinsic \'etale, namely for any neighbourhood $U\subset T_1$ we have that the morphism $(\mathcal{O}_{T_2}(U'),\mathcal{O}^+_{T_2}(U'))\rightarrow (\mathcal{O}_{T_1}(U),\mathcal{O}^+_{T_1}(U))$ is localized intrinsic \'etale. 	
\end{definition}

\begin{definition}
For sousperfectoid adic spaces, a morphism $S_1\rightarrow S_2$ is called localized intrinsic \'etale if locally on $S_1$ this is localized intrinsic \'etale, namely for any neighbourhood $U\subset S_1$ we have that the morphism $(\mathcal{O}_{S_2}(U'),\mathcal{O}^+_{S_2}(U'))\rightarrow (\mathcal{O}_{S_1}(U),\mathcal{O}^+_{S_1}(U))$ is localized intrinsic \'etale. 	
\end{definition}

\begin{definition}
For diamantine adic spaces, a morphism $D_1\rightarrow D_2$ is called localized intrinsic \'etale if locally on $D_1$ this is localized intrinsic \'etale, namely for any neighbourhood $U\subset D_1$ we have that the morphism $(\mathcal{O}_{D_2}(U'),\mathcal{O}^+_{D_2}(U'))\rightarrow (\mathcal{O}_{D_1}(U),\mathcal{O}^+_{D_1}(U))$ is localized intrinsic \'etale. 	
\end{definition}

\

\section{Properties}

\indent We now study the corresponding properties of the corresponding localized intrinsic \'etale morphisms, following \cite{EGAIV4} and \cite{Hu1}. We now assume the corresponding analyticity of the adic rings.

\begin{conjecture}
Any localized intrinsic \'etale morphism of strongly sheafy adic spaces is locally a composition of rational localization and finite \'etale morphism.	
\end{conjecture}

\indent Here is the special situation.

\begin{proposition}
As in \cite{Hu1}, namely in the strongly noetherian situation we have the conjecture holds.	
\end{proposition}

\begin{proof}
This is because in that setting our definition in the intrinsic setting coincides with the more algebraic one in \cite{Hu1}. And note that in this setting the affinoidness of the morphism reduces to just being admitting surjections from Tate algebra over the source.	
\end{proof}


\indent If this is true then we have:

\begin{corollary}
Any localized intrinsic \'etale morphism of sousperfectoid adic spaces is locally a composition of rational localization and finite \'etale morphism. Any localized intrinsic \'etale morphism of diamantine adic spaces is locally a composition of rational localization and finite \'etale morphism.	
\end{corollary}

\begin{proposition}
Compositions of localized intrinsic \'etale morphisms of strongly sheafy adic spaces are again localized intrinsic \'etale morphism.	
\end{proposition}

\begin{proof}
Locally it is the corresponding compositions of topologically finite type morphism, and locally it is the corresponding compositions of the corresponding affinoid morphisms, and locally it is the corresponding compositions of morphisms giving rise to the quasi-isomorphic to zero truncated cotangent complex.	
\end{proof}

\begin{corollary}
Compositions of localized intrinsic \'etale morphisms of sousperfectoid adic spaces are again localized intrinsic \'etale morphism. Compositions of localized intrinsic \'etale morphisms of diamantine adic spaces are again localized intrinsic \'etale morphism.		
\end{corollary}

\begin{proposition}
The localized intrinsic \'etaleness of any morphism $T_1\rightarrow T_2$ of strongly sheafy adic rings is preserved under the base change along any morphism of $T_3\rightarrow T_2$.
\end{proposition}

\begin{proof}
The base change of any morphism of topologically finite type is again of topologically finite type. The affinoidness of morphism is also preserved under any base change morphism. Finally for the cotangent complex locally, we definitely have the corresponding result as well.	
\end{proof}

\begin{proposition}
The \'etale property of a morphism between strongly sheafy rings could be detected locally at each point.	
\end{proposition}

\begin{proof}
Straightforward.	
\end{proof}

\indent We now consider some functoriality issue in our current situation. Now we consider the corresponding localized intrinsic \'etale morphisms under the construction of Witt vectors. Now let:
\begin{align}
A\rightarrow B	
\end{align}
be a general morphism in positive characteristic. Therefore we can take the corresponding Witt vector construction:
\begin{align}
W(A^\flat)\rightarrow W(B^\flat),	
\end{align}
where we assume that $A^\flat\rightarrow B^\flat$ is localized intrinsic \'etale. Here we take the completion if needed along the corresponding Fontainisation.

\begin{proposition}
The map 
\begin{align}
W(A^\flat)\rightarrow W(B^\flat),	
\end{align}
is affinoid if the kernel is closed \footnote[1]{This is again due to the very subtle point around the sheafiness such as in \cite[Theorem 1.4.20]{Ked1}.}.	
\end{proposition}

\begin{proof}
We only need to check this locally, locally we have that there is a lifting:
\begin{align}
W(A^\flat)\{T_1,...\}\rightarrow W(B^\flat)\rightarrow 0	
\end{align}
from the corresponding surjection:
\begin{align}
A^\flat \{\overline{T}_1,...\}\rightarrow B^\flat \rightarrow 0.	
\end{align}
And what we have is that this map on the Witt vector level is also realizing the target as a stably-pseudocoherent module over the source since we have that the target is sheafy (\cite[Theorem 1.4.20]{Ked1}). 
\end{proof}

\begin{proposition}
Same holds for the construction of integral Robba ring $\widetilde{\mathcal{R}}^r_*$ and Robba ring $\widetilde{\mathcal{R}}^{[s,r]}_*$ with respect to closed intervals as in \cite{KL1} and \cite{KL2}.	
\end{proposition}

\

\section{\'Etale-Like Morphisms of $\infty$-Banach Rings and the $\infty$-Analytic Stacks}

\subsection{Approach through De Rham Stacks}

\indent We now extend the corresponding discussion to the $\mathrm{E}_\infty$ objects in \cite[Remark 3.16]{BBBK} by using the ideas as in \cite{R}. Recall from \cite[Theorem 3.14]{BBBK} we have the corresponding categories $\mathrm{Simp}\mathrm{Ind}^m(\mathrm{BanSets}_{H})$ and $\mathrm{Simp}\mathrm{Ind}(\mathrm{BanSets}_{H})$ which are the corresponding categories of the corresponding simplicial sets over the corresponding inductive categories of the corresponding Banach sets over some Banach ring $H$ \footnote{It is safer to assume the open mapping properties on homotopy groups.}.

\begin{theorem}\mbox{\bf{(Bambozzi-Ben-Bassat-Kremnizer)}} The corresponding categories\\
 $\mathrm{Simp}\mathrm{Ind}^m(\mathrm{BanSets}_{H})$ and $\mathrm{Simp}\mathrm{Ind}(\mathrm{BanSets}_{H})$ admit symmetric monoidal model categorical structure. Same holds for $\mathrm{Simp}\mathrm{Ind}^m(\mathrm{NrSets}_{H})$ and $\mathrm{Simp}\mathrm{Ind}(\mathrm{NrSets}_{H})$.	
\end{theorem}

\begin{corollary}
The corresponding categories $\mathrm{Simp}\mathrm{Ind}^m(\mathrm{BanSets}_{H})$ and $\mathrm{Simp}\mathrm{Ind}(\mathrm{BanSets}_{H})$ admit presentations as $(\infty,1)$-categories. Same holds for $\mathrm{Simp}\mathrm{Ind}^m(\mathrm{NrSets}_{H})$ and $\mathrm{Simp}\mathrm{Ind}(\mathrm{NrSets}_{H})$.	
\end{corollary}

\indent Then recall from \cite[Remark 3.16]{BBBK} we have the corresponding ring objects in the $\infty$-categories above:
\begin{align}
\mathrm{sComm}(\mathrm{Simp}\mathrm{Ind}^m(\mathrm{BanSets}_{H})),\\
\mathrm{sComm}(\mathrm{Simp}\mathrm{Ind}(\mathrm{BanSets}_{H})).	
\end{align}
and 
\begin{align}
\mathrm{sComm}(\mathrm{Simp}\mathrm{Ind}^m(\mathrm{NrSets}_{H})),\\
\mathrm{sComm}(\mathrm{Simp}\mathrm{Ind}(\mathrm{NrSets}_{H})).	
\end{align}

\indent Now we use general notation $A$ to denote any object in these categories, regarding as a general $\mathrm{E}_\infty$-ring. We consider the general morphism $A\rightarrow B$ in the first two categories in the following discussion.

\begin{definition}
For any general morphism $A\rightarrow B$, we call this affinoid if we have that that $\pi_0(B)$ is affinoid over $\pi_0(A)$, namely we have that there is a surjection map $\pi_0(A)\left<X_1,...,X_d\right>\rightarrow \pi_0(B)$. And moreover we assume that $\pi_0(B)\otimes_{\pi_0(A)}\pi_n(A)\overset{\sim}{\rightarrow}\pi_n(B)$, for any $n$. 	
\end{definition}

\begin{remark}
Kedlaya's theorem \cite[Theorem 1.4.20]{Ked1} is actually expected to hold in more general setting, at least in the situation where the definition of the affinoid morphisms could be made independent from the corresponding stably-pseudocoherence for open mapping rings (note that we are working over analytic fields). However in the previous definition, we have been not really exact. To be really accurate in the characterization of some desired notion of the affinoidness we think that one has to add certain $\infty$-sheafiness (which certainly holds in \cite{BK}). To be more precise for any general morphism $A\rightarrow B$ in \cite{BK} (namely in current situation one considers the corresponding Banach algebras over the analytic fields in our situation), we call this affinoid if we have that that $\pi_0(B)$ is affinoid over $\pi_0(A)$, namely we have that there is a surjection map $\pi_0(A)\left<X_1,...,X_d\right>\rightarrow \pi_0(B)$. And moreover we assume that $\pi_0(B)\otimes_{\pi_0(A)}\pi_n(A)\overset{\sim}{\rightarrow}\pi_n(B)$, for any $n$. In this situation we have the nice sheafiness (up to higher homotopy). Again similar discussion could be made in the context of \cite{CS}. Note that \cite[Theorem 1.4.20]{Ked1} literally says that the sheafiness is equivalent (in some nice sense but in more flexible derived sense) to the stably-pseudocoherence.
\end{remark}

\begin{definition}
For any general morphism $A\rightarrow B$, we call this localized intrinsic \'etale if we have that that $\pi_0(B)$ is localized intrinsic \'etale over $\pi_0(A)$, namely we have that there is a surjection map $\pi_0(A)\left<X_1,...,X_d\right>\rightarrow \pi_0(B)$ and we have that locally the corresponding truncated topological cotangent complex is basically quasi-isomorphic to zero. And moreover we assume that $\pi_0(B)\otimes_{\pi_0(A)}\pi_n(A)\overset{\sim}{\rightarrow}\pi_n(B)$, for any $n$.	
\end{definition}

\indent We now use the corresponding $X=\mathrm{Spec}A$ to denote the corresponding $\infty$-stack in the opposite categories with respect to the ring $A$. We now define the corresponding de Rham stack attached to $X$ as in \cite[Remark 1.2]{R}:

\begin{definition}
We now define:
\begin{align}
X_\mathrm{dR}(R):=\varinjlim_{I} X(\pi_0(R)/I)	
\end{align}
for any $R$ in 
\begin{align}
\mathrm{sComm}(\mathrm{Simp}\mathrm{Ind}^m(\mathrm{BanSets}_{H})),\\
\mathrm{sComm}(\mathrm{Simp}\mathrm{Ind}(\mathrm{BanSets}_{H})).	
\end{align}
And the injective limit is taking throughout all nilpotent ideals of $\pi_0(R)$. 	
\end{definition}

\begin{remark}
As in \cite[Definition 1.1, Remark 1.2]{R}, one can actually define the corresponding de Rham and crystalline spaces for any functor from $\mathrm{sComm}(\mathrm{Simp}\mathrm{Ind}^m(\mathrm{BanSets}_{H})))	$ and $\mathrm{sComm}(\mathrm{Simp}\mathrm{Ind}(\mathrm{BanSets}_{H})))$ to $\underline{s\mathrm{Sets}} $. This means we do not have to consider $(\infty,1)$-sheaves satisfying certain $\infty$-descent conditions. 
\end{remark}

\begin{definition}
The corresponding formal completion of any morphism $X=\mathrm{Spec}B\rightarrow Y=\mathrm{Spec}A$:
\begin{align}
Y_{X,\mathrm{dR}}	
\end{align}
is defined to be:
\begin{align}
Y_{X,\mathrm{dR}}(R):=\varinjlim_{I} X(\pi_0(R)/I)\times_{Y(\pi_0(R)/I)}	Y(R),
\end{align}
for any $R$ in 
\begin{align}
\mathrm{sComm}(\mathrm{Simp}\mathrm{Ind}^m(\mathrm{BanSets}_{\mathbb{Q}_p})),\\
\mathrm{sComm}(\mathrm{Simp}\mathrm{Ind}(\mathrm{BanSets}_{\mathbb{Q}_p})).	
\end{align}
And the injective limit is taking throughout all nilpotent ideals of $\pi_0(R)$.	
\end{definition}

\begin{remark}
Certainly it is actually not clear how really we should deal with the corresponding ideals here, namely we are not for sure if we need to consider closed ideals. But for simplicial noetherian rings we really have some nice definitions, which will certainly be tangential to the corresponding Huber's original consideration.
\end{remark}

\begin{definition}
We now define the corresponding de Rham intrinsic \'etale morphism to be an affinoid morphism $X=\mathrm{Spec}B\rightarrow Y=\mathrm{Spec}A$ which satisfies the condition:
\begin{align}
\pi_0(X(R))\overset{\sim}{\rightarrow}	\pi_0(Y_{X,\mathrm{dR}}(R)),
\end{align}
for any $\mathrm{E}_\infty$-object $R$.	
\end{definition}

\begin{proposition}
Compositions of de Rham intrinsic \'etale morphisms are again $\mathrm{PD}$ intrinsic \'etale morphisms.
\end{proposition}

\begin{proof}
This is formal.	
\end{proof}

\subsection{Approach through Crystalline Stack and PD-morphisms}

\indent We now use the corresponding $X=\mathrm{Spec}A$ to denote the corresponding $\infty$-stack in the opposite categories with respect to the ring $A$. We now define the corresponding crystalline stack attached to $X$ as in \cite[Definition 1.1]{R}:

\begin{definition}
We now define:
\begin{align}
X_\mathrm{crys}(R):=\varinjlim_{I,\gamma} X(\pi_0(R)/I)	
\end{align}
for any $R$ in 
\begin{align}
\mathrm{sComm}(\mathrm{Simp}\mathrm{Ind}^m(\mathrm{BanSets}_{H})),\\
\mathrm{sComm}(\mathrm{Simp}\mathrm{Ind}(\mathrm{BanSets}_{H})).	
\end{align}
And the injective limit is taking throughout all nilpotent ideals of $\pi_0(R)$ and the corresponding PD-structures. 	
\end{definition}

\begin{definition}
The corresponding PD completion of any morphism $X=\mathrm{Spec}B\rightarrow Y=\mathrm{Spec}A$:
\begin{align}
Y_{X,\mathrm{crys}}	
\end{align}
is defined to be:
\begin{align}
Y_{X,\mathrm{crys}}(R):=\varinjlim_{I,\gamma} X(\pi_0(R)/I)\otimes_{Y(\pi_0(R)/I)}	Y(R),
\end{align}
for any $R$ in 
\begin{align}
\mathrm{sComm}(\mathrm{Simp}\mathrm{Ind}^m(\mathrm{BanSets}_{H})),\\
\mathrm{sComm}(\mathrm{Simp}\mathrm{Ind}(\mathrm{BanSets}_{H})).	
\end{align}
And the injective limit is taking throughout all nilpotent ideals of $\pi_0(R)$ and all the corresponding PD structures.	
\end{definition}

\begin{remark}
Certainly it is actually not clear how really we should deal with the corresponding ideals here and the corresponding PD structures, namely we are not for sure if we need to consider closed ideals. But for simplicial noetherian rings we really have some nice definitions, which will certainly be tangential to the corresponding Huber's original consideration.
\end{remark}

\begin{definition}
We now define the corresponding PD intrinsic \'etale morphism to be an affinoid morphism $X=\mathrm{Spec}B\rightarrow Y=\mathrm{Spec}A$ which satisfies the condition:
\begin{align}
\pi_0(X(R))\overset{\sim}{\rightarrow}	\pi_0(Y_{X,\mathrm{crys}}(R)),
\end{align}
for any $\mathrm{E}_\infty$-object $R$.		
\end{definition}

\begin{proposition}
We have that any de Rham intrinsic \'etale morphism is a PD intrinsic \'etale morphism.	
\end{proposition}

\begin{proof}
This is formal.	
\end{proof}

\begin{proposition}
Compositions of $\mathrm{PD}$ intrinsic \'etale morphisms are again $\mathrm{PD}$ intrinsic \'etale morphisms.
\end{proposition}

\begin{proof}
This is formal.	
\end{proof}

\

\section{Lisse-Like and Non-Ramifi\'e-Like Morphisms of $\infty$-Banach Rings and the $\infty$-Analytic Stacks}

\subsection{Approach through De Rham Stacks}

\indent We now define the corresponding lisse-like morphisms along the idea in the previous section:

\begin{definition}
For any general morphism $A\rightarrow B$, we call this localized intrinsic lisse if we have that that $\pi_0(B)$ is localized intrinsic lisse over $\pi_0(A)$, namely we have that there is a surjection map $\pi_0(A)\left<X_1,...,X_d\right>\rightarrow \pi_0(B)$ and we have that locally the corresponding truncated topological cotangent complex is basically quasi-isomorphic to $\Omega_{\pi_0(B)/\pi_0(A)}[0]$. And moreover we assume that $\pi_0(B)\otimes_{\pi_0(A)}\pi_n(A)\overset{\sim}{\rightarrow}\pi_n(B)$, for any $n$.	
\end{definition}

\begin{definition}
We now define the corresponding de Rham intrinsic lisse morphism to be an affinoid morphism $X=\mathrm{Spec}B\rightarrow Y=\mathrm{Spec}A$ which satisfies the condition:
\begin{align}
\pi_0(X(R))\overset{}{\rightarrow}	\pi_0(Y_{X,\mathrm{dR}}(R))
\end{align}
being surjective, for any $\mathrm{E}_\infty$-object $R$.		
\end{definition}

\begin{definition}
For any general morphism $A\rightarrow B$, we call this localized intrinsic non-ramifi\'e if we have that that $\pi_0(B)$ is localized intrinsic non-ramifi\'e over $\pi_0(A)$, namely we have that there is a surjection map $\pi_0(A)\left<X_1,...,X_d\right>\rightarrow \pi_0(B)$ and we have that locally the corresponding truncated topological cotangent complex is basically quasi-isomorphic to $\Omega_{\pi_0(B)/\pi_0(A)}[0]$ which vanishes as well. And moreover we assume that $\pi_0(B)\otimes_{\pi_0(A)}\pi_n(A)\overset{\sim}{\rightarrow}\pi_n(B)$, for any $n$.	
\end{definition}


\subsection{Approach through Crystalline Stack and PD-morphisms}

\begin{definition}
We now define the corresponding PD intrinsic lisse morphism to be an affinoid morphism $X=\mathrm{Spec}B\rightarrow Y=\mathrm{Spec}A$ which satisfies the condition:
\begin{align}
\pi_0(X(R))\overset{}{\rightarrow}	\pi_0(Y_{X,\mathrm{crys}}(R))
\end{align}
being surjective, for any $\mathrm{E}_\infty$-object $R$\footnote{For non-ramifi\'e situation one considers injectivity.}.		
\end{definition}

\begin{proposition}
We have that any de Rham intrinsic lisse morphism is a PD intrinsic lisse morphism.	
\end{proposition}

\

\section{Perfectization and Fontainisation of $\infty$-Analytic Stacks Situation}

\subsection{Perfectization, Fontainisation and Crystalline Stacks}

\noindent Now we consider the corresponding perfectoidization of $\infty$-analytic stacks after \cite{R} and \cite{Dr1} in the situation where $H$ is assumed to be of characteristic $p$. 

\begin{definition}
For any object in $\infty-\mathrm{Fun}(\mathrm{sComm}\mathrm{Simp}\mathrm{Ind}(\mathrm{Ban}_H),\underline{s\mathrm{Sets}})$, denoted by $X$, we define the corresponding \textit{perfectization} $X^{1/p^\infty}$ of $X$ to be the corresponding functor such that for any $R\in \mathrm{sComm}\mathrm{Simp}\mathrm{Ind}(\mathrm{Ban}_H)$ we have that $X^{1/p^\infty}(R):=X(R^\flat)$ where we define the corresponding tilting \textit{Fontainisation} $R^\flat$ of $R$ to be the corresponding derived completion of:
\begin{align}
\varprojlim\{ ...\overset{\mathrm{Fro}}{\longrightarrow}	R \overset{\mathrm{Fro}}{\longrightarrow} R \overset{\mathrm{Fro}}{\longrightarrow} R\}.
\end{align}

\end{definition}

\indent In the situation of the corresponding monomorphically inductive Banach sets, we have the parallel definition:

\begin{definition}
For any object in $\infty-\mathrm{Fun}(\mathrm{sComm}\mathrm{Simp}\mathrm{Ind}^m(\mathrm{BanSets}_H),\underline{s\mathrm{Sets}})$, denoted by $X$, we define the corresponding \textit{perfectization} $X^{1/p^\infty}$ of $X$ to be the corresponding functor such that for any ring $R\in \mathrm{sComm}\mathrm{Simp}\mathrm{Ind}^m(\mathrm{BanSets}_H)$ we have that $X^{1/p^\infty}(R):=X(R^\flat)$ where we define the corresponding tilting \textit{Fontainisation} $R^\flat$ of $R$ to be the corresponding derived completion of:
\begin{align}
\varprojlim\{ ...\overset{\mathrm{Fro}}{\longrightarrow}	R \overset{\mathrm{Fro}}{\longrightarrow} R \overset{\mathrm{Fro}}{\longrightarrow} R\}.
\end{align}

\end{definition}

\begin{remark}
This is very general notion beyond the corresponding $(\infty,1)$-sheaves satisfying certain descent with respect to the derived rational localizations or more general homotopy Zariski topology as in \cite{BK} and \cite{BBBK}. 	
\end{remark}

\indent Now we follow \cite[Proposition 5.3]{R} and \cite[Section 1.1]{Dr1} to give the following discussion around the corresponding Witt crystalline Stack:

\begin{definition}
We define the corresponding \textit{Witt Crystalline Stack} $X_W$ of any $X$ over $H/\mathbb{F}_p$ in $\infty-\mathrm{Fun}(\mathrm{sComm}\mathrm{Simp}\mathrm{Ind}(\mathrm{Ban}_H),\underline{s\mathrm{Sets}})$ or $\infty-\mathrm{Fun}(\mathrm{sComm}\mathrm{Simp}\mathrm{Ind}^m(\mathrm{Ban}_H),\underline{s\mathrm{Sets}})$ to be the functor $(W(\pi_0(X)^\flat)_{\pi_0(X),\mathrm{crys}})_p^\wedge$. And we define the corresponding pre-crystals to be sheaves of $\mathcal{O}$-modules over this functors when we have that $X$ is an $\infty$-analytic stack with reasonable topology.	
\end{definition}

\begin{example}
For instance if we have that $X=\mathrm{Spa}^h(R)$ coming from the corresponding Bambozzi-Kremnizer spectrum of any Banach ring over $\mathbb{F}_p((t))$ as constructed in \cite{BK}. Then we have that the corresponding functor is $(W(\pi_0(X)^\flat)_{\pi_0(X),\mathrm{crys}})_{p}^\wedge$ is now admitting structures coming from the corresponding homotopy Zariski topology from $X$.	
\end{example}

\begin{example}
For instance if we have that $X=\mathrm{Spec}(\mathbb{F}_p[[t]])$ coming from the corresponding object in the opposite category of $\mathbb{F}_p[[t]]$. Then we have that the corresponding functor is $(W(\pi_0(X)^\flat)_{\pi_0(X),\mathrm{crys}})_{p}^\wedge$ is now admitting structures coming from the corresponding homotopy Zariski topology from $X$, is just the same as the corresponding one in the algebraic setting constructed in \cite[Proposition 5.3]{R} and \cite[Section 1.1]{Dr1}.	
\end{example}

\subsection{Perfectization, Fontainisation and Robba Stacks}

\indent Now we contact \cite{KL1} and \cite{KL2} to look at the corresponding Robba Stacks. Now take any $X$ to be any $\infty$-analytic stack which admits structures of simplicial complete Bornological rings or ind-Fr\'echet structures, namely we have the corresponding complete bornological topology or ind-Fr\'echet topology on $\pi_0(X)$. We work over $H/\mathbb{F}_p$ as well. 

\begin{example}
For instance we take that $X=\mathrm{Spa}^h(R)$ coming from the corresponding Bambozzi-Kremnizer spectrum of any Banach ring over $\mathbb{F}_p((t))$ as constructed in \cite{BK}. 
\end{example}

\begin{definition}
For any $\infty$-analytic stack $X$ as above, we consider the corresponding Witt vector functor $W_n(\pi_0(X)^\flat)$ and then consider $\varinjlim_{n\rightarrow \infty}W_n(\pi_0(X)^\flat)$, namely $W(\pi_0(X)^\flat)$, then we consider the ring $W(\pi_0(X)^\flat)[1/p]$. Then we can take the corresponding completion with respect to the Gauss norm $\|.\|_{\pi_0(X),[s,r]}$ coming from the corresponding norm on $\pi_0(X)$ with respect to some interval $[s,r]\in (0,\infty)$ as in \cite[Definition 4.1.1]{KL2}:
\begin{align}
\widetilde{\Pi}(X)_{[s,r]}:=(W(\pi_0(X)^\flat)[1/p])^\wedge_{\|.\|_{\pi_0(X),[s,r]},\text{Fr\'e}}.	
\end{align}
Then following \cite[Definition 4.1.1]{KL2} we consider the following:
\begin{align}
\widetilde{\Pi}(X)_{r}:=\varprojlim_{s>0}(W(\pi_0(X)^\flat)[1/p])^\wedge_{\|.\|_{\pi_0(X),[s,r]},\text{Fr\'e}}	
\end{align}
and 
\begin{align}
\widetilde{\Pi}(X)_{}:=\varinjlim_{r>0}\varprojlim_{s>0}(W(\pi_0(X)^\flat)[1/p])^\wedge_{\|.\|_{\pi_0(X),[s,r]},\text{Fr\'e}}.	
\end{align}
We call these \textit{Robba stacks}.
\end{definition}
 
\begin{example}
In the situation where $X$ is some $\infty$-analytic stack carrying the corresponding sheaves of simplicial Banach rings (namely not in general bornological or ind-Fr\'echet) we have that the finite projective modules over the three Robba stacks (for by enough intervals) carrying semilinear Frobenius action which realizes the corresponding isomorphisms by Frobenius pullbacks are equivalent. In the noetherian setting we have the same holds for locally finite presented sheaves as well. This is the main results of \cite[Theorem 4.6.1]{KL2}.
\end{example}




\newpage

\subsection*{Acknowledgements} 

This is our independent work, however we would like to thank Professor Kedlaya for helpful discussion, in particular the corresponding nontrivial input on our understanding of the corresponding affinoid morphisms, Hansen-Kedlaya and the corresponding AWS lecture notes.

\newpage

\bibliographystyle{ams}

\end{document}